\def\N{\mathbb N}
\def\Z{\mathbb Z}
\def\R{\mathbb R}
\def\Q{\mathbb Q}
\def\C{\mathbb C}
\def\A{\mathcal A}
\def\B{\mathcal B}
\def\int{\operatorname{int}}
\def\Per{\operatorname{Per}}
\def\per{\operatorname{Per}_\A(\beta)}
\def\pfz{\begin{proof}}
\def\pfk{\end{proof}}
\documentclass[12pt]{amsart}
\usepackage[lmargin=2.5cm, rmargin=2.5cm,bottom=2.5cm,top=2.5cm]{geometry}
\usepackage{amssymb,amsmath,mathdots}
\usepackage{graphicx}

\usepackage{epstopdf}

\usepackage{subfigure}
\usepackage{comment}
\usepackage{enumerate}
\usepackage{tikz}
\usetikzlibrary{shapes,arrows.meta,automata,calc,positioning}

\newtheorem{theorem}{Theorem}[section]
\newtheorem{proposition}[theorem]{Proposition}
\newtheorem{corollary}[theorem]{Corollary}
\newtheorem{lemma}[theorem]{Lemma}
\newtheorem{definition}[theorem]{Definition}
\newtheorem{example}[theorem]{Example}
\newtheorem{remark}[theorem]{Remark}

\title {Periodic representations in algebraic bases}

\author{V\'\i t\v ezslav Kala}
\address{Charles University, Faculty of Mathematics and Physics, Department of Algebra, Sokolov\-sk\' a 83, 18600 Praha~8, Czech Republic}
\address{University of G\" ottingen, Mathematisches Institut, Bunsenstr.~3-5, D-37073 G\"ottingen, Germany}

\email{vita.kala@gmail.com}

\author{Tom\'a\v s V\'avra}
\address{Charles University, Faculty of Mathematics and Physics, Department of Algebra, Sokolov\-sk\' a 83, 18600 Praha~8, Czech Republic}
\email{vavrato@gmail.com}
\date{\today}
\thanks{The authors were supported by Czech Science Foundation GA\v CR, grant 17-04703Y}

\begin{document}
\begin{abstract}
We study periodic representations in number systems with an algebraic base $\beta$ (not a rational integer). We show that if $\beta$ has no Galois conjugate on the unit circle,
then there exists a finite integer alphabet $\mathcal A$ such that every element of $\Q(\beta)$ admits an eventually periodic representation with base $\beta$ and digits in $\mathcal A$.
\end{abstract}
\maketitle
\allowdisplaybreaks
\section{Introduction}
A well known result by Schmidt~\cite{Schmidt} states that if $\beta>1$ is a Pisot number, then the set of numbers
with eventually periodic (greedy) $\beta$-expansions equals precisely to $\Q(\beta).$ On the other hand,
the only bases allowing eventually periodic $\beta$-expansions of $\Q(\beta)$ are the Pisot and Salem numbers.
However, no Salem base has been proved to posses this property.

In~\cite{BMPV}, Baker, Mas\'akov\'a, Pelantov\'a, and the second author studied the $(\beta,\A)$-representations, i.e., the expressions of the form $\sum_{k\geq-L}^{+\infty} a_k\beta^{-k},$ $a_k\in\A$
 without the greedy condition.
One of the problems studied in~\cite{BMPV} was the following: Given an algebraic base $\beta\in\C$, $|\beta|>1$, is there a finite alphabet of digits $\A\subset\Q(\beta)$,
such that
\begin{equation}\label{de:per}
\Q(\beta)=\per :=\big\{ \sum_{k\geq-L}^{+\infty} a_k\beta^{-k}\ : \ a_k\in\A,\ (a_k)_{k\geq -L} \text{ is eventually periodic}\big\}.
\end{equation}
This property indeed holds for the Pisot numbers, and their complex analogy. Our main result is
 the following theorem.
\begin{theorem}\label{thm:intro}
  Let $\beta$ be an algebraic number such that $|\beta|>1$, and let $|\beta'|\neq 1$ for each Galois conjugate $\beta'$ of $\beta$. Then
  there exists $\A\subset\Z$ finite such that $\per=\Q(\beta).$
\end{theorem}
Note that this is a stronger version of Theorem~25 of \cite{BMPV}. An additional condition was required there, namely that
$\tfrac1a\in\Z[\beta,\beta^{-1}]$, where $a$ is the leading coefficient of the minimal polynomial of $\beta$ over $\Z$.
This additional condition is not typically satisfied; the full classification of such bases is also given in~\cite{BMPV}.

We will see that our result follows from the fact that $\tfrac1n\in\per$ for all $n\in\N$.
To obtain this statement we first prove a generalized form of the Fermat's little theorem. Namely, for any $\beta$ algebraic, and for any $n\in\N$,
we show the existence of $i,j\in\Z$ such that
$\beta^i-\beta^j \in n\Z[\beta]$, see Theorem~\ref{thm:main}. This is surprisingly non-trivial for an algebraic non-integer
base.

An important role in our proofs is played by the parallel addition algorithms, see~\cite{FrPeSv}. We use those algorithms
for the reduction of the alphabet of representations while preserving periodicity.
This approach fails if $\beta$ has a conjugate on the unit circle where we cannot make use of
parallel addition algorithms.
Nevertheless, we will show that $\tfrac1n\in\per$ for all $n\in\N$ even in these cases, see Section~\ref{sec:salem}.
We are unable, however, to use our technique to extend the periodicity to the whole $\Q(\beta)$.

The technique used to prove the main result is rather non-constructive. However, we provide a tool
allowing the computation of a $(\beta,\A)$-representation for any element of $\Q(\beta)$ with certain alphabet of digits. These results are contained in Section~\ref{sec:graphs}.

\section{Preliminaries}
Let us provide precise definitions of the basic notations mentioned in the introduction.
\begin{definition}
Let $\beta\in\C$, $|\beta|>1$, and let $\A\subset\C$ be a finite set containing $0$. An expression
\[
x = \sum_{k= -L}^{+\infty} a_k\beta^{-k},\quad a_i\in\A
\]
is called a $(\beta,\A)$-representation of $x$.
\end{definition}

One of the possible constructions of $(\beta,\A)$-representation is the following one given by Thurston in~\cite{thurston}.
Assume we have a set $V\subset\C$ such that $\beta V \subseteq\bigcup_{a\in\A}(V+a),$ then all the elements of $V$ have
a $(\beta,\A)$-representation of the form $\sum_{k=1}^{+\infty} a_k\beta^{-k}$. Moreover, as a consequence,
every $x\in\bigcup_{n\in\N} \beta^{n}V$ has a $(\beta,\A)$-representation. Here we can see that if
$0\in\operatorname{int}(V)$, then there exist $(\beta,\A)$-representations for all the real (or complex) numbers. 
It has been shown that the assumption of each element of $\Q(\beta)$
having an eventually periodic $(\beta,\A)$-representation arising
from such a construction is very restrictive. In particular,
if $\beta\in\R$, then necessarily $|\beta|$ is a Pisot number, see~\cite{BMPV}.
More on finding $(\beta,\A)$-representations can be found for example in~\cite{renyi, ItoSadahiro,AkiSch,DarKat, HaFuIt}.

Given two $(\beta,\A)$-representations, one can study their behaviour under elementary arithmetic operations.
In~\cite{FrHePeSv,FrPeSv}, the authors proved that if $\beta$ has no conjugates on the unit circle, then there exists
$\A\subset \Z$ such that $(\beta,\A)$-representations allow a parallel addition algorithm defined as follows.

\begin{definition}
For a base $\beta\in\C$, $|\beta|>1$, and an alphabet $\A\subset\C$, denote $\B=\A+\A$. We say that $(\beta,\A)$ allows parallel addition if there exist $t,r\in\N$ and $\Phi:\B^{t+r+1}\to\A$ such that
\begin{itemize}
\item $\Phi(0^{t+r+1})=0$;
\item For every $x=\sum_{k\in\Z}x_k\beta^{-k}$ with $x_k=0$ for $k< L$ for some $L$ and $x_k\in\B$, it holds that
$x=\sum_{k\in\Z}z_k\beta^{-k}$, where $z_k=\Phi(x_{k-t}\cdots x_kx_{k+1}\cdots x_{k+r})\in\A$.
\end{itemize}
\end{definition}

\begin{theorem}[\cite{FrPeSv},\cite{FrHePeSv}]
Let $\beta\in\C$, $|\beta|>1$. Then there exists an alphabet $\A\subset\C$ such that $(\beta,\A)$ allows parallel addition if and only if $\beta$ is an algebraic number and $|\beta'|\neq 1$ for every conjugate $\beta'$ of $\beta$.
Moreover, we can choose $\A = \{-M,\dots, 0,\dots ,M\}\subset\Z.$
\end{theorem}

We can thus see a parallel addition algorithm as an algorithm that reduces representations over some large (but finite) alphabet into a $(\beta,\A)$-representations using only a
bounded neighbourhood of each digit. Therefore it rewrites finite (eventually periodic) representations over $\Z$ into finite (eventually periodic) $(\beta,\A)$-representations. The following
statement is thus an easy corollary.

\begin{corollary}[\cite{BMPV}]\label{c:1}
Let $\beta\in\C$, $|\beta|>1$, and let $\A$ be a symmetric alphabet such that $(\beta,\A)$ allows parallel addition. Let $\per$ be as in~\eqref{de:per}, and let
\[
{\rm Fin}_\A(\beta) = \big\{\sum_{k\in I} a_k\beta^{-k} :  a_k\in\A,\ I\subset\Z\text{ is finite}\big\}.
\]
Then
\begin{enumerate}



\item\label{en:3} ${\rm Fin}_\A(\beta)\cdot{\rm Per}_\A(\beta)\subset {\rm Per}_\A(\beta)$;

\item\label{en:4} ${\rm Fin}_\A(\beta) = \Z[\beta,\beta^{-1}]$.
\end{enumerate}
\end{corollary}
In the rest of the text we denote by $\Z_n$ the factorring $\Z/ n\Z$.
\section{The main theorem}
The following proposition was one of the ingredients used in the proof of Theorem~25 of \cite{BMPV}. Since the proposition appeared as a part of a proof, we include it here for completeness.

\begin{proposition}\label{prop:main}
Let $\beta>1$ have no conjugate on the unit circle. Then the existence of $\A$ such that $\tfrac1n\in{\rm Per}_\A(\beta)$
is equivalent to the existence of $i>j\in\Z$ such that $\beta^i-\beta^j\in n\Z[\beta].$
\end{proposition}
\begin{proof}
Suppose that $\tfrac1n$ has an eventually periodic $(\beta,\A)$-representation
\[
\frac1n = \sum_{k\geq -L}^{+\infty} a_k\beta^{-k},\quad\text{with } a_{n}=a_{n+p} \text{ for } n > N.
\]

By summing the period as a geometric series we obtain
$$
\frac1n =\frac{z_1}{\beta^N} + \frac{z_2}{\beta^{N+p}(\beta^p-1)}\quad\text{for some }z_1,z_2\in\Z[\beta],
$$
which can be easily rewritten in the desired form.

Assume that $\beta^i-\beta^j\in q\Z[\beta]$ with $i>j$. Then we have that
\begin{equation}\label{eq:neco}
\frac1n = z\cdot \frac1{\beta^j}\cdot\frac1{\beta^{i-j}-1}, \quad\text{for some } z\in\Z[\beta].
\end{equation}
Then by Corollary~\ref{c:1} we know that
$z\cdot\frac1{\beta^j}\in\mathrm{Fin}_\A(\beta)$. Furthemore,
$$\frac1{\beta^{i-j}-1}=-\sum_{k=0}^\infty{\beta^{-k(i-j)}}\in\mathrm{Per}_\A(\beta),$$
 thus the expression~\eqref{eq:neco} belongs to $\mathrm{Fin}_\A(\beta)\cdot\mathrm{Per}_\A(\beta)\subset\mathrm{Per}_\A(\beta).$
\end{proof}

\begin{theorem}\label{theo:coprime}
Let $\beta$ be an algebraic number with the minimal polynomial $m(x)=a_dx^d+\dots+a_1x+a_0\in\Z[x]$, and let $n\in\N$.
If $\gcd(a_d,n)=\gcd(a_0,n)=1$, then $\beta^i-1\in n\Z[\beta]$ for some $i\in\N.$
\end{theorem}
\begin{proof}
Let $m(x)\in\Z[x]$ be the minimal polynomial of $\beta.$ Define sequences $z^{(k)}=(z_0^{(k)},\dots,z_{d-1}^{(k)})\in \Z_n^d$ and $p_k\in\Z_n$ for $k\in\N_0$ by the relation
\begin{equation}\label{eq:1}
\sum_{i=0}^{d-1} z_i^{(k)}x^i + p_km(x) \equiv \sum_{i=0}^{d-1}z_i^{(k+1)} x^{i+1}\pmod {n\Z[x]}\quad\text{with}\quad z^{(0)} = (1,0,\dots,0).
\end{equation}
The sequences $\{z^{(k)}\}$ and $\{p_k\}$ are uniquely defined. Indeed, each $z^{(k)}$ enforces that $p_k:=-z^{(k)}_0m(0)^{-1}\pmod n$
(using the invertibility of $m(0)=a_0$),
and consequently $z^{(k+1)}$ is also defined. Conversely, note that to each $z^{(k+1)}$ there are unique $z^{(k)}$ and $p_k$, as $a_d$ is invertible modulo $n$.

Clearly, the sequence $\{z^{(k)}\}$ can take only finitely many values. Let $r$ be the smallest index such that $z^{(r)}=z^{(s)}$ for some $s<r.$
Then because of the uniqueness of the followers and predecessors in $\{z^{(k)}\}$ we have $z^{(r)}=z^{(0)}=(1,0,\dots,0)$, whence necessarily
$z^{(r-d+1)} = (0,\dots,0,1).$
Then by~\eqref{eq:1} we have
\begin{equation}\label{eq:2}
\sum_{k=0}^{r-d} x^kp_km(x)\equiv x^{r}-1\pmod {n\Z[x]},
\end{equation}
because $z^{(0)}=(1,0,\dots,0),\ z^{(r)}=(0,\dots,0,1),$ and the rest of the
summands cancel out. We obtain the statement by applying
the homomorphism $x\mapsto\beta: \mathbb Z[x]\rightarrow \Z[\beta]$ on~\eqref{eq:2}.
\end{proof}

We will need the following lemma to generalize Theorem~\ref{theo:coprime}
also to the cases $\gcd(a_d,n)\neq 1$ or/and $\gcd(a_0,n)\neq 1.$ We will denote
the leading coefficient of the minimal polynomial of $\beta$ over $\Z$ as $c(\beta)$.
\begin{lemma}\label{lem:minpol}
Let $m_i(x)$ be the minimal polynomial for $\beta^i$, and let $c_i=c(\beta^i)$ be the leading coefficient of $m_i(x)$.
If $p\mid c_1$ for a prime $p$, then for each $j$ there is $i$ such that $p^j \mid c_i$.
\end{lemma}

\begin{proof}
Let us first assume that $c_1=p^k$ for some $k$, and that there is $j$ such that $p^j\nmid c_i$ for each $i$.

Observe that if $b\beta$ is an algebraic integer for some $b\in\Z$, then $c(\beta)\mid b$. Also $c(\beta)\beta$ is always an algebraic integer.
Hence $p^k\beta$ is an integer, and so also $p^{ik}\beta^i$ is an integer. Thus $c_i\mid p^{ik}$ and $c_i$ is a power of $p$. Since $p^j\nmid c_i$, it is at most $(j-1)$st power. Since $c_i\beta^i$ is an integer, we see that also $p^{j-1}\beta^i$ is an integer for each $i$.

We conclude that $\Z[\beta]\subset\frac 1{p^{j-1}}\mathcal O_K$ (where $K=\Q(\beta)$).
But $\frac 1{p^{j-1}}\mathcal O_K$ is a finitely generated $\Z$-module and $\Z$ is noetherian, and so also $\Z[\beta]$ is finitely generated $\Z$-module. But this implies that $\beta$ is an algebraic integer.
This is a contradiction with $c_1>1$.

For the general case, let $a_1=p^k b$ with $p\nmid b$ and consider $\gamma:=b\beta$. Let $b_i:=c(\gamma^i)$; by the assumption we have that $b_1=p^k$
(because a candidate for the minimal polynomial for $\gamma$ is $b^{d-1}m(x/b)$; we may need to divide by the gcd of coefficients, but not by $p$, because $p$ was coprime to $b$).

By the first part of the proof, we know that for each $j$ there is $i$ such that $p^j\mid b_i$. Let now $n_i(x)$ be the minimal polynomial for $\gamma^i=b^i\beta^i$.
Then $n_i(b^ix)$ is a candidate for the minimal polynomial for $\beta^i$; we may need to divide by the gcd of coefficients, but again not by $p$, as it is coprime to $b^i$. Thus $p^j\mid b_i$ implies that $p^j\mid a_i$.
\end{proof}
Note that we get the same statement for constant coefficients just by considering $\beta^{-1}$.

Now we are ready to prove the main technical result of this section, which can be viewed as a generalized version of little Fermat's theorem.
There is a number of these in the literature, including some that do not require $\beta$ to be an algebraic integer (e.g., Chapter~23 in~\cite{hasse}), but we could not locate our version, which seems to be in a somewhat different vein than the others.

\begin{theorem}\label{thm:main}
Let $\beta$ be an algebraic number of degree $d$. Then for each $n\in\N$ there exists $i>j\in\Z$ such that $\beta^i-\beta^j\in n\Z[\beta]$.
\end{theorem}

\begin{proof}
Let us first prove that if the statement holds for coprime $n_1,n_2\in\N$, then it also holds for $n_1n_2$.
Assuming its validity for $n_1, n_2$,  we can suppose that
  $1-\beta^{i_1}\in n_1\Z[\beta,\beta^{-1}]$, and $1-\beta^{i_2}\in n_2\Z[\beta,\beta^{-1}].$
  Then it is also true that for all $k\in\N$ we have $1-\beta^{ki_1}\in n_1\Z[\beta,\beta^{-1}]$ and $1-\beta^{ki_2}\in n_2\Z[\beta,\beta^{-1}]$.
  From $\gcd(n_1,n_2)=1$ it follows that $1-\beta^{i_1i_2}\in n_1n_2\Z[\beta,\beta^{-1}],$ i.e., the statement is true for $n=n_1n_2$.

Hence it remains to prove the theorem when $n=p^\ell$
  for a prime $p$ and $\ell\in\N.$
  The case $\gcd(a_0,p)=\gcd(a_d,p)=1$ is solved in Theorem~\ref{theo:coprime}.
  Thus assume w.l.o.g. that $p|a_d$, otherwise we consider $\beta^{-1}$. We will proceed by induction on the degree of $\beta$.
  According to Lemma~\ref{lem:minpol} we can find $k\in\N$ such that $\beta^k$ has the minimal polynomial $m_k(x)$, such that $n$ divides its leading coefficient $c(\beta^k)$. Roots of $m_k(x)-c(\beta^k)x^d$ are of a smaller degree, therefore by the induction we have that
  \[{p}(x)(m_k(x)-c(\beta^k)x^d)=x^i-x^j-nz(x)\]
  for some $p(x),z(x)\in\Z[x]$.
  Then after a simple rearrangement, and under the map $x\mapsto \beta^k$ we obtain $$\beta^{ki}-\beta^{kj}=nz(\beta^k)-\beta^{kd}c(\beta^k){p}(\beta^k)\in n\Z[\beta^k]\subset n\Z[\beta].$$
  The induction is complete by realizing that the statement is true for $\beta\in\Z.$
  \end{proof}

  \begin{proof}[Proof of Theorem~\ref{thm:intro}]
  Each $x\in\Q(\beta)$ can be written as $x = \tfrac zn$ with $z\in\Z[\beta]$, and $n\in\N.$ 
  By Theorem~\ref{thm:main} together with Proposition~\ref{prop:main} we have that $\tfrac1n\in\operatorname{Per}_\A(\beta).$
  Then by Corollary~\ref{c:1} we have that $z\in\operatorname{Fin}_\A(\beta)$, and subsequently also that $x\in\operatorname{Per}_\A(\beta).$
  \end{proof}
\section{Bases with conjugates on the unit circle}\label{sec:salem}
In the previous section, a necessary tool for obtaining our results was an existence of the parallel addition in base $\beta.$
The reason was that we were then able to convert eventually periodic representations over an infinite alphabet $\Z$ to a finite one.
However, the existence of the parallel algorithms is possible only if there is no conjugate of $\beta$ lying on the unit circle. Nevertheless,
finding periodic representations of $\tfrac1n$ is possible if one proceeds more carefully. In this section we prove the following theorem.
\begin{theorem}\label{thm:salem}
Let $\beta$ be an algebraic number such that $|\beta'|=1$ for a conjugate $\beta'$. Then $\tfrac1n\in\Per_\A(\beta)$ for some $\A\subset\Z$ finite.
\end{theorem}

\begin{lemma}\label{lem:salem}
Let $\beta$ be an algebraic number. Then for each $n\in\N$ one can find $i(n)>j(n)\in\Z$ such that
$\beta^{i(n)}-\beta^{j(n)} = n\sum_{k=0}^{m(n)} d_k(n)\beta^k = z(n)$, such that
\begin{enumerate}
 \item $m(n) < 2(i(n)-j(n))$,
 \item there exists $C>0$, such that $|d_k(n)|<C$ for any $k,n.$
\end{enumerate}
\end{lemma}
\begin{proof}
Fix an $n\in\N$, the existence of $i>j$ and $d_k$'s satisfying $\beta^i-\beta^j = n\sum_{i=0}^{m}d_k\beta^k$ is given by Theorem~\ref{thm:main}.
Multiplying both sides of this equation by $\beta^{r(i-j)}$ and summing for $r=0,1,\dots,s$ we obtain
\[
\beta^{(s+1)i-sj}-\beta^j=n \sum_{k=0}^{(s+1)i-sj+m} \widetilde{d_k}\beta^k.
\]
We can satisfy item (1) by setting $i(n) = (s+1)i-sj$, $j(n)=j$, $d_k(n) = \widetilde{d_k}$, and $m(n) = (s+1)i-sj+m$ for an appropriate
value $s$.

Thus we have
\[x^{i(n)}-x^{j(n)}-n\sum_{k=0}^{m(n)} {d_k}(n)x^k = p_n(x)m(x)\]
for some $p_n(x)\in\Z[x].$ We can reduce the coefficients of $p_n(x)$ modulo $n$ to assume that they are all between 0 and $n$.
Let $M$ be the maximum of absolute values of coefficients of $m(x)$.
Then the polynomial $p_n(x)m(x)$ has all coefficients less than $nM(\deg m+1)$ in absolute value. Consequently, we see that in item (2) we can take $C=M(\deg m+1)$.
\end{proof}
\begin{remark}
Note that one can replace the factor $2$ in item (1) by $1+\varepsilon$ for any $\varepsilon>0$.
\end{remark}
\begin{proof}[Proof of Theorem~\ref{thm:salem}]
For fixed $n\in\N$ apply Lemma~\ref{lem:salem}, i.e., we have that $\beta^i-\beta^j = nz$ for some $i>j$ and $z=\sum_{k=0}^{m}d_k(n)\beta^k\in\Z[\beta]$. Then
\[
\frac1n = \frac{1}{\beta^i}\frac{z}{1-\beta^{j-i}} = \frac1{\beta^i}\sum_{k=0}^{+\infty}z\beta^{-k(i-j)}.
\]
The latter is indeed a periodic representation over an integer alphabet that is bounded by $2\max\{|d_k(n)| : 0\leq k\leq m\}$ (here we used that $m<2(i-j)$). Since $d_k(n)$ are bounded independently
of $n$, we can choose a common alphabet for all $n\in\N.$
\end{proof}
\section{Computational point of view}\label{sec:graphs}
The results so far showed the existence of the eventually periodic $(\beta,\A)$-representations. Because of the induction, the proof of Theorem~\ref{thm:main} does not give
an explicit way of finding $i,j$ such that $\beta^i-\beta^j\in n\Z[\beta].$ In this section we show how to compute the pair $i,j$. The method we use is in fact in the
background of the proof of Theorem~\ref{theo:coprime}.

From now on we will handle elements of $\Z[\beta]$ as elements of the quotient ring $\Z[x]/{m(x)}\cong \Z[\beta]$, where $m(x)$ is the minimal polynomial of $\beta,$ through the isomorphism $x\mapsto \beta.$

Let us start with an example which is not covered bny the results of~\cite{BMPV}.
\begin{example}\label{ex:algo1}
Consider $m(x)=3x^2+2x+3$, and $n=6.$ Then we have
 $$0\equiv (2x^2+3x+4)m(x)=6x^4+13x^3+24x^2+17x+12\pmod {m(x)},$$
hence $$x^3-x \equiv -6x^4-12x^3-24x^2-18x-12\in n(\Z[x]/(m(x)),$$
or equivalently, $$\beta^3-\beta = -6\beta^4-12\beta^3-24\beta^2-18\beta-12\in n\Z[\beta].$$
\end{example}
Let us show how such $i,j$ can be found in general. Assume that $x^i-x^j\equiv np(x)$ in $\Z[x]/m(x)$ for some $p(x)\in\Z[x]$. This is equivalent to the existence of $r(x)\in\Z[x]$ such that
\begin{equation}\label{eq:xyz}
x^i-x^j-qp(x) = r(x)m(x)\quad\text{in }\Z[x].
\end{equation}
The product $r(x)m(x)$ can be viewed (roughly speaking) as a ``sum of shifted multiples of $m(x)$''. This idea is illustrated in the following table, where we continue with Example~\ref{ex:algo1}.

$$
\begin{array}{rccccc}
2x^2m(x)=&6&4&6& & \\
3xm(x)=& &9&6&9& \\
4m(x)=& & &12&8&12 \\ \hline
(2x^2+3x+4)m(x)=&6&13&24&17&12
\end{array}
$$
In each row lies a multiple of the minimal polynomial, the power of $x$ corresponds to the shift. In order to satisfy~\eqref{eq:xyz} for some $p(x)$, we want the tuple of the sums of the columns (in our case $(6,13,24,17,12)$) to be equivalent $\pmod n$ to a vector with the only two non-zero entries being $1$ and $-1$. In fact, we can also consider the table to live in $\Z_n$ to directly obtain the result
$$
\begin{array}{rccccc}
2x^2m(x)=&0&4&0& & \\
3xm(x)=& &3&0&3& \\
4m(x)=& & &0&2&0 \\ \hline
(x^2+2x+1)m(x)=&0&1&0&-1&0
\end{array}
$$
When constructing $r(x)$, we can proceed from higher powers of $x$ to lower (or from left to right in the table) wanting to add an appropriate multiple of the minimal polynomial such that the left most digit sums to zero (or $1$ at one position and $-1$ at another position) during each step. However, we do not have prior knowledge of where the digits $1$ and $-1$ should be created.

\begin{definition}
Let $m(x)=\sum_{i=0}^d a_ix^i$ and let $n\in\N.$ The graph $G(m,n)=(V,E)$ is the oriented graph with vertices $V = \Z_n^d\times\{A,B,C\}$, and with the set $E$ of labeled edges $(y_d,\dots, y_1;\gamma)\xrightarrow{k}(z_d,\dots,z_1;\delta), k\in\Z_n,$ if
\begin{enumerate}
\item $\gamma=\delta$ and $\sum_{i=1}^d y_ix^i+k\sum_{i=0}^d{a_ix^i} \equiv \sum_{i=1}^{d}z_ix^{i-1}\quad\pmod{ n\Z[x]},$
\item $\gamma=A,\delta=B$ and $\sum_{i=1}^d y_ix^i+k\sum_{i=0}^d{a_ix^i} \equiv x^d+\sum_{i=1}^{d}z_ix^{i-1}\ \pmod {n\Z[x]},$
\item $\gamma=B,\delta=C$ and $\sum_{i=1}^d y_ix^i+k\sum_{i=0}^d{a_ix^i} \equiv -x^d+\sum_{i=1}^{d}z_ix^{i-1}$
    $\pmod {n\Z[x]}.$
\end{enumerate}

\end{definition}
The graph $G(m,n)$ has the following meaning. Consider again Example~\ref{ex:algo1}. The labels of edges correspond to the coefficients of $r(x)=2x^2+3x+4,$ i.e.,\ we have a path
$$
(0,0;A)\xrightarrow{2}(4,0;A)\xrightarrow{3}(0,3;B)\xrightarrow{4}(5,0;B)\xrightarrow{0}(0,0;C)
$$
Note that the change of the third entry of a vertex from $A$ to $B$ corresponds to the situation that we created the digit $1$, while the change from $B$ to $C$ means that the digit $-1$ was produced.

\begin{theorem}\label{thm:algo}
Let $\beta$ be an algebraic number with no conjugate on the unit circle, let $m(x)$ be the minimal polynomial of $\beta$, and let $n\in\N$.
Then $\tfrac1n\in{\rm Per}_\A(\beta)$ for some $\A\subset\C$ if and only if in the graph $G(m,n)$ there exists a path from $(0,\dots,0;A)$ to $(0,\dots,0;C)$.

Moreover, if this path has labels $c_0,c_1,\dots,c_{s-1}$, then
$$(c_0x^{s-1}+c_1x^{s-2} + \dots + c_{s-1})m(x)\equiv x^i-x^j\quad\pmod {n\Z[x]}$$
for some $i,j\in\Z.$
\end{theorem}
\begin{proof}
Let $$(z^{(0)},\gamma^{(0)})\xrightarrow{c_0}(z^{(1)},\gamma^{(1)})\xrightarrow{c_1}\dots\xrightarrow{c_{s-1}}(z^{(s)},\gamma^{(s)}),$$
where $z^{(k)}=(z^{(k)}_d,\dots,z^{(k)}_1)$, $z^{(0)}=z^{(s)}=(0,\dots,0)$ and $\gamma^{(0)}=A,\gamma^{(s)}=C$ be a path in $G(m,n)$.

According to the definition of $G(m,n)$ we have that
\begin{equation}\label{eq:cesta}
c_{k}m(x)\equiv\alpha_k x^d - \sum_{i=1}^dz_i^{(k)}x^i+\sum_{i=1}^dz_i^{(k+1)}x^{i-1}\quad\pmod {n\Z[x]},
\end{equation}
where
$$
\alpha_k=\begin{cases} 1&\text{if } \gamma^{(k)}=A, \gamma^{(k+1)}=B,\\ -1&\text{if } \gamma^{(k)}=B, \gamma^{(k+1)}=C, \\ 0&\text{otherwise.}\end{cases}
$$

Then by multiplying~\eqref{eq:cesta} by $x^{s-k-1}$, and summing for each $k=0,\dots,s-1$, we obtain
$$m(x)\sum_{k=0}^{n-1}x^{s-k-1}c_k\equiv\sum_{k=0}^{s-1}x^{s-k-1}\alpha_k\quad\pmod {n\Z[x]}.$$
The right side is of this form because $z^{(0)}=z^{(s)}=(0,\dots,0)$, and the rest of the summands cancel out.

Thus we have $c(x)m(x)=x^i-x^j+np(x)$ with $c(x),p(x)\in\Z[x]$, $i=s-k_1-1$ and $j=s-{k_2}-1$ if $(z^{(k_1)},A)\xrightarrow{k_1}(z^{(k_1+1)},B)$ and $(z^{(k_2)},B)\xrightarrow{k_2}(z^{(k_2+1)},C).$ Hence $x^i-x^j\equiv-np(x)\pmod{\Z[x]/m(x)}$, and using the isomorphism $\Z[x]/m(x)\rightarrow\Z[\beta]$ we obtain $\beta^i-\beta^j=-np(\beta)\in n\Z[\beta].$
\end{proof}

For a given base $\beta$, set $\A$ such that $(\beta,\A)$ allows parallel addition. Then computing an eventually periodic $(\beta,\A)$-representation of $x:=\tfrac zn\in\Q(\beta)$, $z\in\Z[\beta]$ can be done by the following steps:
\begin{enumerate}
  \item construct the graph $G(m,n)$, and find $i,j\in \Z$, and $z\in\Z[\beta]$ such that $\beta^i-\beta^j = nz$ using Theorem~\ref{thm:algo};
  \item  construct an eventually periodic $(\beta,\A)$-representation of $\tfrac1n$ as
    \[
	\frac 1n = -\frac {\widetilde z}{\beta^j}\sum_{k=0}^\infty{\beta^{-k(i-j)}}
\]
(see the proof of Proposition~\ref{prop:main});
  \item use a parallel addition algorithm (see~\cite{FrPeSv}) to reduce the digits of the eventually periodic $(\beta,\A)$-representation 
\[x = -\frac {z\widetilde z}{\beta^j}\sum_{k=0}^\infty{\beta^{-k(i-j)}}\]
into the digit alphabet $\A$.
\end{enumerate}
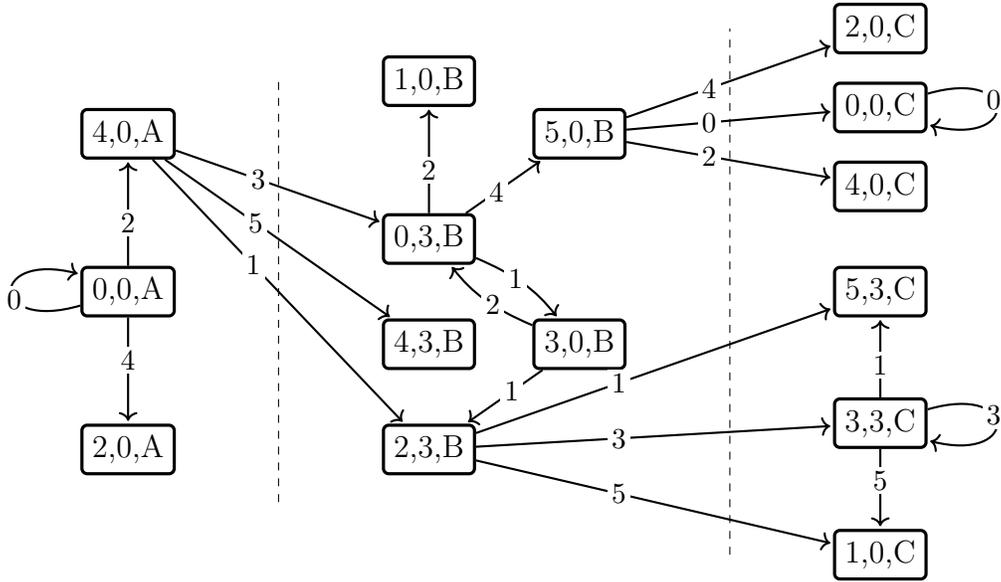
\begin{figure}[h]
\tikzset{
	nn/.style = {rectangle, draw, rounded corners=2pt, very thick},
	nnA/.style = {nn, fill=yellow},
	nnB/.style = {nn, fill=green},
	nnC/.style = {nn, fill=cyan},
	ee/.style = {thick,->,shorten >=1pt},
	e1/.style = {in=160, out=200, min distance=1.3cm},
	e2/.style = {in=340, out=20, min distance=1.3cm},
	e3/.style = {bend left=15},
}
\begin{tikzpicture}[yscale=0.7]
\foreach \a/\b/\c/\x/\y in {%
	0/0/A/-4/0, 4/0/A/-4/3, 2/0/A/-4/-3,
	2/3/B/0/-3, 4/3/B/0/-1, 0/3/B/0/1, 1/0/B/0/4, 5/0/B/2/3, 3/0/B/2/-1,
	1/0/C/6/-5, 3/3/C/6/-2.5, 5/3/C/6/0, 4/0/C/6/2, 0/0/C/6/3.5, 2/0/C/6/5%
	} {
	\node[nn] (\a\b\c) at (\x,\y) {\a,\b,\c};
}
\foreach \abc/\aabbcc/\l/\s in {%
	00A/40A/2/,
	00A/00A/0/e1,
	00C/00C/0/e2,
	03B/30B/1/e3,
	00A/20A/4/,
	40A/03B/3/,
	40A/43B/5/,
	03B/10B/2/,
	40A/23B/1/,
	03B/50B/4/,
	30B/23B/1/,
	50B/20C/4/,
	50B/00C/0/,
	50B/40C/2/,
	23B/53C/1/,
	23B/33C/3/,
	23B/10C/5/,
	30B/03B/2/e3,
	33C/10C/5/,
	33C/33C/3/e2,
	33C/53C/1/%
	}{
	\draw (\abc) edge[ee, \s] node[pos=0.4,fill=white,font=\small,inner sep=1.5pt] {\l} (\aabbcc);
}
\draw [dashed] (-2,-4) -- (-2, 4) (4, -5) -- (4, 5);
\end{tikzpicture}
\caption{The graph $G(m,n)$ for $m(x) = 3x^2+2x+3,\ n=6$ as in Example~\ref{ex:algo1}.}
\end{figure}

\bibliographystyle{plain}
\bibliography{Biblio.bib}
\end{document}